[1994/12/01]

\documentclass[draft]{amsart}

\usepackage{amsmath, amsfonts, amssymb}
\usepackage{amsthm,amscd,latexsym,mathrsfs,hyperref}
\usepackage{bm}

\usepackage{color}

\theoremstyle{plain} 
\newtheorem{theorem}{Theorem}[section]
\newtheorem{corollary}[theorem]{Corollary}
\newtheorem{lemma}[theorem]{Lemma}
\newtheorem{proposition}[theorem]{Proposition}

\theoremstyle{definition}
\newtheorem{definition}{Definition}[section]
\newtheorem{remark}{Remark}[section]

\DeclareMathOperator{\sat}{\mathrm{sat}}
\DeclareMathOperator{\hypo}{\mathrm{hypo}}

\DeclareMathOperator{\Lat}{Lat}

\def\X{{\mathcal X}}

\begin{document}

\title{A note on real operator monotone functions}

\author{Marcell Ga\'al and Mikl\'os P\'alfia}

\address{Marcell Ga\'al
\newline  \indent R\'enyi Institute of Mathematics \newline \indent Hungarian Academy of Sciences,
\newline \indent Budapest, Re\'altanoda utca 13-15,  1053 Hungary}
\email{marcell.gaal.91@gmail.com}

\address{
Mikl\'os P\'alfia
\newline  \indent Department of Mathematics, Sungkyunkwan University, 
\newline \indent Suwon 440-746, Korea}
\email{palfia.miklos@aut.bme.hu} 

\address{and
\newline \indent Bolyai Institute, Interdisciplinary Excellence Centre \newline \indent University of Szeged \newline  \indent  Szeged, Aradi v\'ertan\'uk tere 1, 6720 HUNGARY}

\date{\today}

\maketitle

\begin{abstract}
In this paper we initiate the study of real operator monotonicity for functions of tuples of operators, which are multivariate structured maps with a functional calculus called free functions that preserve the order between real parts (or Hermitian parts) of bounded linear Hilbert space operators. We completely characterize such functions on open convex free domains in terms of ordinary operator monotone free functions on self-adjoint domains. Further assuming the more stringent free holomorphicity, we prove that all such functions are affine linear with completely positive nonconstant part. This problem has been proposed by David Blecher at the biannual OTOA conference held in Bangalore in December 2016.
\end{abstract}

\maketitle

\section{Introduction}
In the seminal papers of Blecher et al. \cite{blecherRead2,blecherRead,blecherRuan}, research on operators $X\in\mathcal{B}(\mathcal{H})$ with positive Hermitian parts
\begin{equation*}
\Re X=\frac{1}{2}(X+X^*)\geq 0
\end{equation*}
on a Hilbert space $\mathcal{H}$, which are called \emph{real positive operators}, has been initiated to study general operator algebras. 
Such operators are also called \emph{accretive} and play an essential role in strongly continuous operator semigroups, see for instance in \cite{yosida}. Among others, further studies related to real positivity can be found in \cite{bearden,blecher,drury}. In particular, motivated in part by the paper of Kubo-Ando \cite{kubo} characterizing two-variable operator means of positive bounded linear operators on a Hilbert space, Blecher and Wang in \cite{blecher} studied root functions and the extension of the Pusz-Woronowitz geometric mean \cite{pusz}, which is given by the formula
\begin{equation*}
A\#B=\max
\left\{ X\geq 0 ~:~ \begin{bmatrix}
A & X \\
X & B
\end{bmatrix}\right\},
\end{equation*}
to the real positive setting. It is fundamental that the mean $\#$ preserves the positive definite order induced by the positive cone of bounded linear operators, and it also satisfies the arithmetic-geometric-harmonic mean inequalities \cite{bhatia,bhatiaholbrook}. It has been pointed out in \cite{blecher} that even though the usual formula
\begin{equation*}
A\#B=A^{1/2}\left(A^{-1/2}BA^{-1/2}\right)^{1/2}A^{1/2}
\end{equation*}
makes sense for two real positive operators $A,B$, it does not preserve the real positive definite order and the corresponding arithmetic-geometric-harmonic mean inequalities also fail badly.

Since, it has been somewhat surprising that no nontrivial nice real positive order preserving functions were known, even though the purely self-adjoint counterpart, the theory of (free) operator monotone functions with respect to the positive definite order is well understood in the single variable case by the classical theory of Loewner \cite{lowner,hansen3}, and now also in the non-commutative multivariable case \cite{agler,palfia2,pascoe,pascoe2} that exists in the realm of free function theory \cite{verbovetskyi}. The theory of operator means has also been extended to cover a large class of functions of probability measures on positive operators endowed with the stochastic order \cite{palfia1}. 

It is obvious that the affine (arithmetic mean like) function
\begin{equation*}
F(X_1,X_2)=cI+aX_1+bX_2
\end{equation*}
for scalars $a,b\geq 0$ and $c\in\mathbb{C}$ preserves the real positive order, and apparently so do its multivariate analogues. 
We prove that essentially no other locally bounded similarity invariant free function preserves the real positive order on open domains. As a precursor to this, we show that even if we consider possibly non-continuous free functions $F(X_1,\ldots,X_k)$ that are invariant just under unitary conjugations, then $F$ is real operator monotone if and only if the real part $\Re F(X_1,\ldots,X_k)$ of such functions is an operator monotone function of the real part $(\Re X_1,\ldots,\Re X_k)$ of the variables $(X_1,\ldots,X_k)$ such that $\Re F$ is independent of the skew-Hermitian part (imaginary part) of $(X_1,\ldots,X_k)$. These results show that real operator monotonicity is a rather strong property, especially rigid in the class of holomorphic functions. We further demonstrate, how our analysis generalizes to the more general case of free functions with a domain that is a free open subset of $\mathcal{B}(\mathcal{H})\otimes\mathcal{Z}$ for an arbitrary operator space $\mathcal{Z}$, not just $\mathcal{Z}=\mathbb{C}^k$ which corresponds to the set of $k$-tuples of operators above.

The paper is organized as follows.
In the second section we briefly review some necessary background material on free sets and free function theory, and also on real positivity along with basic characterizations of real monotonicity.
Then we establish intimate connections between real monotonicity and concavity with respect to the real positive definite order in Sections 3-4.
Finally, Section 5 deals with the complete characterization of real operator monotone free functions.

\section{Real monotonicity}
A bounded linear operator $X\in\mathcal{B}(E)$ is \emph{real positive} (denoted by $X\geq_{\mathrm{Re}} 0$) whenever the real part of $X$ is positive semi-definite, that is
\begin{equation*}
\Re(X):=\frac{X+X^*}{2}\geq 0
\end{equation*}
on the Hilbert space $E$. The symbol $\mathbb{P}_{\mathrm{Re}}(E)$ stands for the cone of real positive definite operators over the Hilbert space $E$ such that their real parts are invertible.
For $A,B\in\mathcal{B}(E)$ we write $A\leq_{\mathrm{Re}} B$ if and only if $A-B\geq_{\mathrm{Re}} 0$. The \emph{real positive order} for $k$-tuples $A,B\in\mathcal{B}(E)^k$ is defined similarly as $A\leq_{\mathrm{Re}} B$ exactly when $A_i-B_i\geq_{\mathrm{Re}} 0$ for all $i\in \mathbb{N}_k$. Note however that $\leq_{\mathrm{Re}}$ is just a preorder, it is not a partial order because the conditions $A\leq_{\mathrm{Re}}B$ and $B\leq_{\mathrm{Re}} A$ do not imply that $A=B$. They just imply $\Re A=\Re B$.

\begin{definition}[Free set and matrix convex set]\label{def:matrix_covexity}
A collection $(D(E))$ of sets of operators $D(E)\subseteq\mathcal{B}(E)^k$ for each Hilbert space $E$ is a called a \emph{free set} whenever for all Hilbert space $E,K$ we have the following:
\begin{itemize}
	\item[1)] $U^*D(E)U\subseteq D(K)$ for all unitary $U:E\mapsto K$.
	\item[2)] $D(E)\oplus D(K)\subseteq D(E\oplus K)$
\end{itemize}
where $U^*XU:=(U^{*}X_1U,\ldots,U^{*}X_kU)$ for $X\in\mathcal{B}(E)^k$.

If additionally (2) holds for any linear isometry $U:K\mapsto E$, then $(D(E))$ is a \emph{matrix convex set}.
\end{definition}

We remark that if a given free set $(D(E))$ is matrix convex, then according to \cite{helton4} each $D(E)$ is convex in the usual sense.

\begin{definition}[Free function]\label{D:freeFunction}
Let $\mathcal{L}$ be a fixed Hilbert space. A multivariate function $F:D(E)\mapsto \mathcal{B}(\mathcal{L}\otimes E)$ for a domain $D(E)\subseteq\mathcal{B}(E)^k$ defined for all Hilbert spaces $E,K$ is called a \textit{free function} whenever for all $A \in \mathcal{B}(E)^k$ and $B\in \mathcal{B}(K)^k$ in the domain of $F$, we have
\begin{itemize}
	\item[1)] unitary invariance, that is
	$$F(U^{*}A_1U,\ldots,U^{*}A_kU)=(I_\mathcal{L}\otimes U^{*})F(A_1,\ldots,A_k)(I_\mathcal{L}\otimes U)$$ 
	holds for all unitaries $U\in \mathcal{B}(E)$;
	\item[2)] direct sum invariance, that is
	$$F\left(A_1 \oplus B_1,\ldots,A_k \oplus B_k\right)=F(A_1,\ldots,A_k) \oplus F(B_1,\ldots,B_k).$$
\end{itemize}
\end{definition}
Notice that the above extends naturally the notion of a free function given as a graded map between self-adjoint sets \cite{palfia2,pascoe,pascoe2}.

\begin{definition}[Real monotonicity and concavity]\label{D:realmon}
\end{definition}
\begin{itemize}
    \item[1)] Given a free set $(B(E))$ where $B(E)\subseteq \mathcal{B}(E)^k$, a free function $F:D(E)\mapsto \mathcal{B}(\mathcal{L}\otimes E)$ is said to be \emph{real operator monotone} if we have $F(A)\leq_{\mathrm{Re}} F(B)$, whenever $A\leq_{\mathrm{Re}} B$ for $A,B\in D(E)$.
    \item[2)] If each $D(E)$ is convex, then the free function $F:D(E)\mapsto \mathcal{B}(\mathcal{L}\otimes E)$ is said to be \emph{real operator concave} if for all $A,B\in D(E)$ and $\lambda\in[0,1]$, we have
\begin{equation*}
(1-\lambda)F(A)+\lambda F(B)\leq_{\mathrm{Re}} F((1-\lambda)A+\lambda B).
\end{equation*}
    \item[3)] If one of the above two properties is satisfied only for finite dimensional $E$, then we say that the free function $F:D(\mathbb{C}^n)^k\mapsto \mathcal{B}(\mathcal{L}\otimes \mathbb{C}^n)$ is \emph{real $n$-monotone} or \emph{real $n$-concave}, accordingly.
\end{itemize}

\medskip

Let $\mathbb{S}(E):=\{X\in\mathcal{B}(E):X^*=X\}$ denote the set of self-adjoint bounded linear operators acting on a Hilbert space $E$.

\begin{remark}
If a free domain $(D(E))$ consists of only self-adjoint operators, then a real operator monotone free function $F:D(E)\mapsto \mathbb{S}(\mathcal{L}\otimes E)$ is operator monotone in the usual sense, that is, it preserves the positive definite order. For such functions a powerful structure theory is already available for matrix convex $(D(E))$ with nonempy interior in \cite{agler,palfia2,pascoe2}. They are essentially analytic functions of its entries such that they analytically continue to upper half-spaces, that is, operator entries with strictly positive imaginary parts. In \cite{palfia2} a widely applicable formula, based on the Schur complement, is also available through which the analytic extension can be obtained.
\end{remark}

\begin{proposition}\label{P:FrechetRealPos}
Let $D(E)$ be a matrix convex set with $D(E)\subseteq\mathcal{B}(E)^k$ and let $F:D(E)\mapsto \mathcal{B}(\mathcal{L}\otimes E)$ be a free function such that the Frech\'et-derivative $DF(X)[H]$ exists for any $H\in\mathcal{B}(E)^k$ and $X\in D(E)$. Then $F$ is real operator monotone if and only if $DF(X)[\cdot]$ is a real completely positive linear map, that is, we have
\begin{equation}\label{eq:P:FrechetRealPos}
DF(X)[H]\geq_{\mathrm{Re}}0
\end{equation}
for $H\geq_{\mathrm{Re}}0$.
\end{proposition}
\begin{proof}
$"\Rightarrow":$ Let $X\in D(E)$ and $0\leq_{\mathrm{Re}}H\in\mathcal{B}(E)^k$. Then $X+tH\geq_{\mathrm{Re}} X$ for any $t>0$, hence $F(X+tH)\geq_{\mathrm{Re}}F(X)$. This implies that
\begin{equation*}
DF(X)[H]=\lim_{t\to 0+}\frac{F(X+tH)-F(X)}{t}\geq_{\mathrm{Re}} 0.
\end{equation*}
Further $DF(X)[H]$ is also a free function of its variables $(X,H)$. Thus, the linear map $H\mapsto DF(\cdot)[H]$ satisfies the amplification formula
 of completely bounded linear maps (for a proof, see Proposition 2.10. \cite{pascoe}), that is
\begin{equation*}
DF(X\otimes I)[H\otimes V]=DF(X)(H)\otimes V
\end{equation*}
for any $V\in\mathbb{S}(K)$, thus $DF(X)[\cdot]$ is also completely positive.

$"\Leftarrow":$ Let $A\leq_{\mathrm{Re}}B\in D(E)$ and $A(t):=(1-t)A+tB$ for $t\in[0,1]$. Then $A'(t)=B-A\geq_{\mathrm{Re}} 0$ and it follows that $DF(A(t))[A'(t)]\geq_{\mathrm{Re}} 0$ by the assumption. Since
\begin{equation*}
\int_{0}^{1}DF(A(t))[A'(t)]dt=F(B)-F(A),
\end{equation*}
we get that $F(B)\geq_{\mathrm{Re}} F(A)$.
\end{proof}

\begin{remark}\label{R:Steinspring}
It is known that all real completely positive linear maps satisfy the same Stinespring representation formula 
 as completely positive linear maps do \cite[Theorem 2.4.]{bearden}
\end{remark}

\section{Characterizations of real operator monotone functions on $\mathbb{P}_{\mathrm{Re}}$}

In this section we turn to the investigation of general properties of real operator monotone and concave functions. 
We shall need the following technical lemma, which is a slight modification of \cite[Lemma 3.5.5.]{niculescu}
\begin{lemma}\label{L:concave-bounded}
Let $F$ be a concave function into $\mathbb{S}(E)$ on an open convex set $U$ in a normed linear space. If $F$ is bounded from below in a neighborhood of one point of $U$, then $F$ is locally bounded on $U$.
\end{lemma}
\begin{proof}
Suppose that $F$ is bounded from below by $MI$ for some $M\in\mathbb{R}$ on an open ball $B(a,r)$ with radius $r$ around $a$. Let $x\in U$ and choose $\rho>1$ such that $z:=a+\rho(x-a)\in U$. If $\lambda=1/\rho$, then $$V=\{v:v=(1-\lambda)y+\lambda z, y\in B(a,r)\}$$ 
is a neighborhood of $x=(1-\lambda)a+\lambda z$, with radius $(1-\lambda)r$. Moreover, for $v\in V$ we have
\begin{equation*}
F(v)\geq (1-\lambda)F(y)+\lambda F(z)\geq (1-\lambda)MI+\lambda F(z)\geq KI
\end{equation*}
for some $K\in\mathbb{R}$. To show that $F$ is bounded above in the same neighborhood, choose arbitrarily $v\in V$ and notice that $2x-v\in V$. By the concavity of $F$, one finds that $$F(x)\geq \frac{F(v)+F(2x-v)}{2}$$ which easily yields
\begin{equation*}
F(v)\leq 2F(x)-F(2x-v)\leq 2F(x)-KI.
\end{equation*}
\end{proof}

\begin{proposition}[see also Proposition 3.5.4 in \cite{niculescu}]\label{P:concave-cont}
A function $F:\mathbb{P}_{\mathrm{Re}}(E)^k\mapsto \mathcal{B}(E)$ with concave real part that is locally bounded from below has a continuous real part $\Re(F):\mathbb{P}_{\mathrm{Re}}(E)^k\mapsto \mathbb{S}(E)$ in the norm topology.
\end{proposition}
\begin{proof}
Let $U\subseteq\mathbb{P}_{\mathrm{Re}}(E)^k$ be an open norm bounded neighborhood with respect to the operator norm $\|\cdot\|$. Let $A\in U$ and $r>0$ such that the open ball $$B(A,2r)=\{X\in U:\|X-A\|<2r\}\subseteq U.$$ 
Let $X,Y\in B(A,r)$ and $X\neq Y$ such that $\alpha:=\|Y-X\|<r$. Define
\begin{equation}\label{eq:P:concave-cont-1}
Z:=Y+\frac{r}{\alpha}(Y-X).
\end{equation}
Then
\begin{equation*}
\|Z-A\|\leq\|Y-A\|+\frac{r}{\alpha}\|Y-X\|<2r,
\end{equation*}
that is, $Z\in B(A,2r)$. By \eqref{eq:P:concave-cont-1} we have
\begin{equation*}
Y=\frac{r}{r+\alpha}X+\frac{\alpha}{r+\alpha}Z,
\end{equation*}
so by the real operator concavity of $F$ we get
\begin{equation*}
\Re(F)(Y)\geq\frac{r}{r+\alpha}\Re(F)(X)+\frac{\alpha}{r+\alpha}\Re(F)(Z),
\end{equation*}
which after rearranging yields
\begin{equation*}
\begin{split}
\Re(F)(X)-\Re(F)(Y)&\leq\frac{\alpha}{r+\alpha}(\Re(F)(X)-\Re(F)(Z))\\
&\leq\frac{\alpha}{r+\alpha}2MI\leq\frac{\alpha}{r}2MI,
\end{split}
\end{equation*}
where the real number $M>0$ provides a local bound for $\Re(F)$ on $U$ in the form of 
$$-2MI\leq \Re(F)(X)-\Re(F)(Z)\leq 2MI$$ 
in view of Lemma~\ref{L:concave-bounded}. Now exchange the role of $X$ and $Y$ in the above to obtain the reverse inequality
\begin{equation*}
\Re(F)(Y)-\Re(F)(X)\leq\frac{\alpha}{r}2MI.
\end{equation*}
From the above pair of inequalities we get
\begin{equation*}
\|\Re(F)(Y)-\Re(F)(X)\|\leq2\frac{M}{r}\|Y-X\|
\end{equation*}
proving the continuity.
\end{proof}

A net of operators $\{A_i\}_{i\in \mathcal{I}}$ is called increasing if $A_i\geq A_j$ for $i\geq j$ and $i,j\in\mathcal{I}$. Also $\{A_i\}_{i\in \mathcal{I}}$ is bounded from above if there exists some real constant $K>0$ such that $A_i\leq KI$ for all $i\in\mathcal{I}$. It is well known that any bounded from above increasing net of operators $\{A_i\}_{i\in \mathcal{I}}$ has a least upper bound $\sup_{i\in\mathcal{I}}A_i$ such that $B_j:=A_j-\sup_{i\in\mathcal{I}}A_i$ converges to $0$ in the strong operator topology. Similarly if we have a decreasing net of bounded operators that is bounded from below, then the net converges to its greatest lower bound.

The next characterization result is an extension of Theorem 2.1 in \cite{hansen3} to several variables and to the case of the real positive order. The proof is analogous to that of Theorem 2.1. We consider the finite dimensional situation, however, the proof is presented in such a way that it works also in the infinite dimensional setting as well.
\begin{proposition}\label{P:contmonotone}
Let $F:\mathbb{P}_{\mathrm{Re}}(\mathbb{C}^{2n})^k\mapsto \mathcal{B}(\mathbb{C}^{2n})$ be a real $2n$-monotone function. Then its restriction $F:\mathbb{P}_{\mathrm{Re}}(\mathbb{C}^{n})^k\mapsto \mathcal{B}(\mathbb{C}^{n})$ is real $n$-concave. Moreover, the real part $\Re(F):\mathbb{P}_{\mathrm{Re}}(\mathbb{C}^{n})^k\mapsto \mathbb{S}(\mathbb{C}^{n})$ is norm-continuous.
\end{proposition}
\begin{proof}
Let $A,B\in\mathbb{P}_{\mathrm{Re}}(\mathbb{C}^{n})^k$ and let $\lambda\in[0,1]$. Then the $2n$-by-$2n$ block matrix 
\begin{equation*}
V:=\left[ \begin{array}{cc}
\lambda^{1/2}I_n & -(1-\lambda)^{1/2}I_n \\
(1-\lambda)^{1/2}I_n & \lambda^{1/2}I_n \end{array} \right]
\end{equation*}
is unitary.
Elementary calculation reveals that
\begin{equation*}
V^*\left[ \begin{array}{cc}
A & 0 \\
0 & B \end{array} \right]V=
\left[ \begin{array}{cc}
\lambda A+(1-\lambda)B & \lambda^{1/2}(1-\lambda)^{1/2}(B-A) \\
\lambda^{1/2}(1-\lambda)^{1/2}(B-A) & (1-\lambda)A+\lambda B \end{array} \right].
\end{equation*}
Set $D:=-\lambda^{1/2}(1-\lambda)^{1/2}(\Re(B)-\Re(A))$ and notice that for any given $\epsilon>0$
\begin{equation*}
\left[ \begin{array}{cc}
\lambda A+(1-\lambda)B+\epsilon I & 0 \\
0 & 2Z \end{array} \right]-V^*\left[ \begin{array}{cc}
A & 0 \\
0 & B \end{array} \right]V\geq_{\mathrm{Re}}
\left[ \begin{array}{cc}
\epsilon I & D \\
D & Z \end{array} \right]
\end{equation*}
if $Z \geq (1-\lambda)\Re(A)+\lambda \Re(B)$. The last $k$-tuple of block matrices is positive semi-definite if $Z \geq D^2/\epsilon$. So for sufficiently large positive definite $Z$ we have
\begin{equation*}
V^*\left[ \begin{array}{cc}
A & 0 \\
0 & B \end{array} \right]V\leq_{\mathrm{Re}} \left[ \begin{array}{cc}
\lambda A+(1-\lambda)B+\epsilon I & 0 \\
0 & 2Z \end{array} \right].
\end{equation*}
For such $Z>0$, by the $2n$-monotonicity of $F$ we get
\begin{equation*}
F\left(V^*\left[ \begin{array}{cc}
A & 0 \\
0 & B \end{array} \right]V\right)\leq_{\mathrm{Re}} \left[ \begin{array}{cc}
F(\lambda A+(1-\lambda)B+\epsilon I) & 0 \\
0 & F(2Z) \end{array} \right].
\end{equation*}
We also have that
\begin{equation*}
\begin{split}
&F\left(V^*\left[ \begin{array}{cc}
{A} & 0 \\
0 & {B} \end{array} \right]V\right)=
V^*\left[ \begin{array}{cc}
F({A}) & 0 \\
0 & F({B}) \end{array} \right]V\\
&=\left[ \begin{array}{cc}
\lambda F({A})+(1-\lambda)F({B}) & \lambda^{1/2}(1-\lambda)^{1/2}(F({B})-F({A})) \\
\lambda^{1/2}(1-\lambda)^{1/2}(F({B})-F({A})) & (1-\lambda)F({A})+\lambda F({B}) \end{array} \right],
\end{split}
\end{equation*}
hence we obtain that
\begin{equation}\label{eq:P:concave}
\lambda F({A})+(1-\lambda)F({B})\leq_{\mathrm{Re}} F(\lambda{A}+(1-\lambda){B}+\epsilon I).
\end{equation}

Now since $F$ is real $2n$-monotone, $\Re(F)(X+\epsilon{I})$ for $\epsilon>0$ forms a decreasing net of operators bounded from below by $\Re(F)(X)$, thus the right strong limit
$$\Re(F^{+})({X}):=\inf_{\epsilon>0}\Re(F)(X+\epsilon{I})=\lim_{\epsilon\to 0+}\Re(F)(X+\epsilon{I})$$
exists for all ${X}\in\mathbb{P}_{\mathrm{Re}}(\mathbb{C}^{2n})^k$ defining the real part of $F^{+}$. The imaginary part is defined as $\Im(F^{+})(X):=\Im(F)(X)$. Hence for any $\epsilon>0$, using \eqref{eq:P:concave}, we obtain
\[
\begin{gathered}
\lambda F^{+}({A})+(1-\lambda)F^{+}({B})\leq_{\mathrm{Re}} \lambda F({A}+\epsilon{I})+(1-\lambda)F({B}+\epsilon{I}) \\
\leq_{\mathrm{Re}} F(\lambda{A}+(1-\lambda){B}+2\epsilon{I}).
\end{gathered}
\]
Taking the limit $\epsilon\to 0+$ in the strong operator topology we conclude that
\begin{equation*}
\lambda F^{+}({A})+(1-\lambda)F^{+}({B})\leq_{\mathrm{Re}} F^{+}(\lambda{A}+(1-\lambda){B})
\end{equation*}
meaning that the free function $F^{+}$ is real $n$-concave. Also
\begin{equation*}
F({X})\leq_{\mathrm{Re}} F^{+}({X})\leq_{\mathrm{Re}} F({X}+\epsilon{I})
\end{equation*}
for all $\epsilon>0$, since $\Re(F)$ is monotone increasing. Thus, $\Re(F^{+})$ is bounded from below on order bounded sets, whence by Proposition~\ref{P:concave-cont} $\Re(F^{+})$ is norm continuous on order bounded sets because every point $A\in\mathbb{S}$ has a basis of neighborhoods in the norm topology that are order bounded sets.

As the last step, again by the real monotonicity of $F$ we have
\begin{equation*}
F^{+}({X}-\epsilon{I})\leq_{\mathrm{Re}} F({X})\leq_{\mathrm{Re}} F^{+}({X}),
\end{equation*}
and since $\Re(F^{+})$ is norm-continuous we get $F=F^{+}$ by taking the norm-limit $\epsilon\to 0+$. Hence we can also take the norm-limit $\epsilon\to 0+$ in \eqref{eq:P:concave} and conclude that $F$ is real $n$-concave and $\Re(F)$ is continuous in the norm topology.
\end{proof}

\begin{corollary}\label{P:concavemonotone}
A real operator monotone function $F:\mathbb{P}_{\mathrm{Re}}(E)^k\mapsto \mathcal{B}(E)$ is real operator concave, and it has a norm-continuous real part $\Re(F)$.
\end{corollary}
\begin{proof}
The proof goes along the lines of the previous Proposition~\ref{P:contmonotone}, where the role of $\mathbb{C}^{n}$ is taken by $E$ and using the fact that when $\dim(E)=+\infty$ we have that $E\oplus E\simeq E$.
\end{proof}

The reverse implication is also true if $F$ is bounded from below, its proof goes along the lines of Theorem 2.3 in \cite{hansen3}.
So it is worth to isolate the following result.

\begin{theorem}\label{T:concavemonotone}
Let $F:\mathbb{P}_{\mathrm{Re}}(E)^k\mapsto \mathbb{P}_{\mathrm{Re}}(E)$ be a real operator concave ($n$-concave) function. Then $F$ is real operator monotone ($n$-monotone).
\end{theorem}

\section{Hypographs and convexity}

In this section we will use the theory of matrix convex sets introduced first by Wittstock. For more on free convexity and matrix convex sets the reader is referred to \cite{effros,helton,helton2,helton3,helton4}. 

Let $\Lat(E)$ denote the \emph{lattice of subspaces} of $E$. The notation $K\leq E$ means that $K$ is a closed subspace of $E$, hence a Hilbert space itself.

\begin{definition}
A graded collection $C=(C(K))$, where each $C(K)\subseteq \mathcal{B}(K)^k$, is \emph{closed with respect to reducing subspaces} if for any tuple of operators $(X_1,\ldots,X_k)\in C(K)$ and any corresponding mutually invariant subspace $N\subseteq K$, we have that $(\hat{X}_1,\ldots,\hat{X}_k)\in C(N)$, where all the $\hat{X}_i$'s are the restrictions of $X_i$ to the invariant subspace $N$ for $i\in \mathbb{N}_k$.
\end{definition}

\begin{lemma}[Lemma 2.3 in \cite{helton4}, \S 2 in \cite{helton2}]\label{lem:matrix_convexity_equivalent}
Suppose that the graded collection $C=(C(K))$, where each $C(K)\subseteq \mathcal{B}(K)^k$ respects direct sums in the sense of 1) in Definition~\ref{def:matrix_covexity} and it respects unitary conjugation in the sense of 2) in Definition~\ref{def:matrix_covexity} with $N=K$.
\begin{itemize}
\item[1)] If $C$ is closed with respect to reducing subspaces, then $C$ is matrix convex if and only if each $C(K)$ is convex in the classical sense of taking scalar convex combinations.
\item[2)] If $C$ is nonempty and matrix convex, then  $0=(0,\ldots,0)\in C(1)$ if and only if $C$ is closed with respect to simultaneous conjugation by contractions.
\end{itemize}
\end{lemma}

Given a set $A\subseteq\mathcal{B}(E)$ we define its \emph{saturation} as $$\sat(A):=\{X\in\mathcal{B}(E) ~:~ \exists Y\in A, Y\geq_{\mathrm{Re}} X\}.$$
Similarly, for a graded collection $C=(C(K))$, where each $C(K)\subseteq\mathcal{B}(K)$, its \emph{saturation} $\sat(C)$ is the disjoint union of $\sat(C(K))$ for each Hilbert space $K$.

\begin{definition}[Hypographs]
Let $F:D(E)\mapsto \mathcal{B}(E)$ be a free function where $(D(E))$ is a free set. Then we define its \emph{real hypograph} $\hypo_{\mathrm{Re}}(F)$ as the graded collection of the saturation of its image, that is
$$\hypo_{\mathrm{Re}}(F)=(\hypo_{\mathrm{Re}}(F)(K)):=(\{(Y,X)\in\mathcal{B}(K)\times D(K):Y\leq_{\mathrm{Re}} F(X)\}).$$
\end{definition}

\begin{theorem}\label{T:convex_hypographs}
Let $F:D(E)\mapsto \mathcal{B}(E)$ be a free function, where $(D(E))$ is a matrix convex set which is closed with respect to reducing subspaces. Then its real hypograph $\hypo_{\mathrm{Re}}(F)$ is a matrix convex set if and only if $F$ is real operator concave.
\end{theorem}
\begin{proof}
Suppose first that $F$ is real operator concave. We will prove the matrix convexity of $\hypo_{\mathrm{Re}}(F)$ by establishing the properties in (1) of Lemma~\ref{lem:matrix_convexity_equivalent}. By the definition of real concavity, and the convexity of $\mathbb{P}_{\mathrm{Re}}$ and the real order intervals, it follows easily that for each Hilbert space $K$, $\hypo_{\mathrm{Re}}(F)(K)$ is convex in the usual sense of taking scalar convex combinations. To see that $\hypo_{\mathrm{Re}}(F)$ is closed with respect to reducing subspaces, assume that $(Y,X)\in\hypo_{\mathrm{Re}}(F)(L)$ such that $(Y,X)=(\hat{Y},\hat{X})\oplus(\overline{Y},\overline{X})$ and $$(\hat{Y},\hat{X})\in\mathcal{B}(K)\times D(K), \quad (\overline{Y},\overline{X})\in\mathcal{B}(N)\times D(N)$$ 
for Hilbert spaces $K\oplus N=L$. Then since $F$ is a free function, it respects direct sums, hence 
$$Y\leq_{\mathrm{Re}} F(X)=F(\hat{X})\oplus F(\overline{X}).$$ Again by the definition of free functions, we have $F(\hat{X})\in\mathcal{B}(K)$ and $F(\overline{X})\in\mathcal{B}(N)$. Since $Y=\hat{Y}\oplus\overline{Y}$, it follows that $\hat{Y}\leq_{\mathrm{Re}} F(\hat{X})$ and $\overline{Y}\leq_{\mathrm{Re}} F(\overline{X})$, or in another words $(\hat{Y},\hat{X})\in\hypo_{\mathrm{Re}}(F)(K)$ and $(\overline{Y},\overline{X})\in\hypo_{\mathrm{Re}}(F)(N)$.

As for the converse, suppose that $\hypo_{\mathrm{Re}}(F)$ is a matrix convex set. First notice that $\hypo_{\mathrm{Re}}(F)$ is closed with respect to reducing subspaces. Indeed, similarly to the above assume that $(Y,X)\in\hypo_{\mathrm{Re}}(F)(L)$ with $(Y,X)=(\hat{Y},\hat{X})\oplus(\overline{Y},\overline{X})$ and $$(\hat{Y},\hat{X})\in\mathcal{B}(K)\times D(K), \quad (\overline{Y},\overline{X})\in\mathcal{B}(N)\times D(N)$$ for Hilbert spaces $K\oplus N=L$. Then since $F$ is a free function, it respects direct sums, hence $$Y\leq_{\mathrm{Re}} F(X)=F(\hat{X})\oplus F(\overline{X}).$$ Again by the definition of free functions, we have $F(\hat{X})\in\mathcal{B}(K)$ and $F(\overline{X})\in\mathcal{B}(N)$. Since $Y=\hat{Y}\oplus\overline{Y}$, it follows that $\hat{Y}\leq_{\mathrm{Re}} F(\hat{X})$ and $\overline{Y}\leq_{\mathrm{Re}} F(\overline{X})$, that is $(\hat{Y},\hat{X})\in\hypo_{\mathrm{Re}}(F)(K)$ and $(\overline{Y},\overline{X})\in\hypo_{\mathrm{Re}}(F)(N)$. So again by part 1) of Lemma~\ref{lem:matrix_convexity_equivalent} it follows that for each Hilbert space $L$, $\hypo_{\mathrm{Re}}(F)(L)$ is convex in the usual sense. It means that for all $t\in[0,1]$ and $A,B\in\mathbb{P}_{\mathrm{Re}}(L)^k$ we have that the tuple $$(Y,X):=(1-t)(F(A),A)+t(F(B),B)$$ lies in $\hypo_{\mathrm{Re}}(F)(L)$, that is $$(1-t)F(A)+tF(B)\leq_{\mathrm{Re}} F(X)=F((1-t)A+tB)$$ meaning that $F$ is real operator concave.
\end{proof}

The above Theorem~\ref{T:convex_hypographs} combined with Theorem~\ref{T:concavemonotone} leads to the following.

\begin{corollary}\label{C:convex_hypographs_monotone}
Let $F:\mathbb{P}_{\mathrm{Re}}(E)^k\mapsto \mathbb{P}_{\mathrm{Re}}(E)$ be a free function. Then its hypograph $\hypo_{\mathrm{Re}}(F)$ is a matrix convex set if and only if $F$ is real operator monotone.
\end{corollary}

\section{Representation and rigidity of real operator monotone functions}

In this section we establish some further characterizations of real operator monotone free functions in terms of operator monotone free functions. This will imply by \cite{palfia2,pascoe2} that the real parts of such functions must be analytic with respect to the real parts of their variables. Further rigidity is derived if we assume free holomorphicity for the function $F$, which according to \cite{verbovetskyi} is equivalent to a mild local boundedness condition on $F$ along with that in Definition~\ref{D:freeFunction} property 1) is strengthened to cover invariance by similarities, that is
\begin{itemize}
	\item[1')] $F(S^{-1}A_1S,\ldots,S^{-1}A_kS)=(I_\mathcal{L}\otimes S^{-1})F(A_1,\ldots,A_k)(I_\mathcal{L}\otimes S)$ for every invertible $S:E\mapsto K$.
\end{itemize}

\begin{theorem}\label{T:DependsOnFirstVar}
Let $(D(E))$ be a free domain where $D(E)\subseteq\mathcal{B}(E)^k$ is defined for all Hilbert spaces $E$. Let $F:D(E)\mapsto\mathcal{B}(E)$ be a free function. Define the free function $\Re F:\Re D(E)\times \Im D(E)\mapsto\mathbb{S}(E)$ by the decomposition $F(X)=\Re F(\Re X,\Im X)+i\Im F(\Re X,\Im X)$. Then $F$ is real operator monotone if and only if $\Re F$ is independent of its second variable $\Im X$ and it is operator monotone in its first variable $\Re X$.
\end{theorem}
\begin{proof}
Assume first that $F$ is real operator monotone. Let $X\in D(E)$ and let $W\in\Im D(E)\subseteq\mathbb{S}(E)$ be arbitrary. Then $\Re X+iW\leq_{\mathrm{Re}}\Re X+i\Im X\leq_{\mathrm{Re}}\Re X+iW$, so by the real monotonicity
\begin{equation*}
F(\Re X,W)\leq_{\mathrm{Re}}F(\Re X,\Im X)\leq_{\mathrm{Re}}F(\Re X,W)
\end{equation*}
where $W$ is arbitrary for any $X$. Hence we conclude that $\Re F:\Re D(E)\times \Im D(E)\mapsto\mathbb{S}(E)$ is independent of its second variable. By the real operator monotonicity of $F$, it then follows that its real part $\Re F$ is operator monotone in its first variable as a map of self-adjoint operators into self-adjoint operators, it is also not difficult to check that it respects direct sums and simultaneous unitary conjugations, whence a free function itself.

For the converse assume that $\Re F(\Re X,\Im X)=G(\Re X)$ where $G:\Re D(E)\mapsto\mathbb{S}(E)$ is a free operator monotone function. Then clearly $F:D(E)\mapsto\mathcal{B}(E)$ is real operator monotone.
\end{proof}

An immediate consequence is the following representation.

\begin{corollary}\label{C:DependsOnFirstVar}
Let $(D(E))$ be a free domain where $D(E)\subseteq\mathcal{B}(E)^k$ is defined for all Hilbert spaces $E$. Let $F:D(E)\mapsto\mathcal{B}(E)$ be a free function. Then $F$ is real operator monotone if and only if
\begin{equation}\label{eq:C:1}
F(\Re X,\Im X)=G(\Re X)+\mathrm{i}H(\Re X,\Im X)
\end{equation}
where $H:\Re D(E)\times \Im D(E)\mapsto\mathbb{S}(E)$ is a free function and $G:\Re D(E)\mapsto\mathbb{S}(E)$ is an operator monotone free function.
\end{corollary}

\begin{remark}
It is clear that for a free function $F$, its imaginary part $\Im F$ does not have any influence on the real operator monotonicity of $F$. 
It can be arbitrary and thus there are real operator monotone functions which are not free holomorphic or analytic. However their real part is always analytic or even holomorphic as a free function of self-adjoint operators, see characterizations of free operator monotonicity in \cite{palfia1}.
\end{remark}

From this point on, we shall assume that $\dim(E)<\infty$ in all statements, in order to avoid delving deeply into topological subtleties. Given a free function $F:D(E)  \mapsto \mathcal{B}(E)$ on a free set $(D(E))$ we say that it is also \emph{free holomorphic} if it satisfies (1') and for each norm continuous linear functional $h:\mathcal{B}(E)\to \mathbb{C}$ the multivariable complex valued function $h(F(X))$ is holomorphic, or equivalently G\^{a}teaux-differentiable, see \cite{verbovetskyi}. Notice that (1') forces the free domain $(D(E))$ to be closed under simultaneous similarity transformations as well, not just simultaneous unitary conjugations. Also we note again that according to the main results in \cite{verbovetskyi}, for a free function $F$, (1') and a mild local boundedness condition on $F$ implies that $F$ is free holomorphic.

\begin{theorem} \label{T:main}
Given a free set $(\X(E))$, let $F:\X(E)  \to \mathcal{B}(E)$ be a free holomorphic function where each $\X(E)\subseteq \mathcal{B}(E)^k$ is open. Then $F$ is real operator monotone if and only if it admits an expression
\[
F(X)=a_0\otimes I+ \sum_{j=1}^k a_j\otimes X_j
\]
where $a_j\in \mathbb{C}$, with $a_j \geq 0$ for $j\in \mathbb{N}_k$.
\end{theorem}

We emphasize in advance that the result is still new in the single variable case as well.
Its proof rests heavily on an auxiliary lemma concerning multivariate complex functions.

\begin{definition}[Pluriharmonic function]
Let $\Omega \subseteq \mathbb{C}^m$ be a complex domain for some $m\in \mathbb{N}$.
A function $u:\Omega \mapsto \mathbb{C}$ is called \emph{pluriharmonic} whenever for any complex line
\[
L_{a,b}:=\{ a+bz : z\in \mathbb{C} \}
\]
formed by every couple of complex tuples $a,b\in \mathbb{C}$ the function $z\mapsto f(a+bz)$ is harmonic on the segment $\Omega \cap L_{a,b}$.
\end{definition}
Introducing the Wirtinger derivatives
\[
\frac{\partial}{\partial z}=\frac{1}{2}\left(\frac{\partial}{\partial x}-\mathrm{i}\frac{\partial}{\partial y}\right), \quad
\frac{\partial}{\partial \overline{z}}=\frac{1}{2}\left(\frac{\partial}{\partial x}+\mathrm{i}\frac{\partial}{\partial y}\right)
\]
and their corresponding multivariate counterparts
\[
\partial=
\begin{pmatrix}
\frac{\partial}{\partial z_1}\\
\vdots\\
\frac{\partial}{\partial z_m}\\
\end{pmatrix} ,\quad
\overline{\partial}=
\begin{pmatrix}
\frac{\partial}{\partial \overline{z}_1}\\
\vdots\\
\frac{\partial}{\partial \overline{z}_m}\\
\end{pmatrix}
\]
the pluriharmonic functions can be characterized by the following system of partial differential equations.
\begin{equation} \label{pluriharmonic}
   \partial \overline\partial u =0 \quad \mbox{ throughout } \Omega.
\end{equation}

\begin{definition}[Levi form]
The \emph{Levi form} associated to a $\mathcal{C}^2(\Omega)$ function at the footpoint $z\in \mathbb{C}$ is the Hermitian form
\[
\mathcal{L}(z;c,d)=\sum_{j=1}^m\sum_{k=1}^m \frac{\partial^2u(z)}{\partial z_{j}\partial \overline{z}_k}c_j\overline{d}_k
\]
for any $c,d\in \mathbb{C}^m$.
\end{definition}

\begin{lemma} \label{L:pluriharmonic}
Assume that the holomorphic function $f:  \Omega \mapsto \mathbb{C}^k$ on a complex domain $\Omega\subseteq\mathbb{C}^m$ admits the form $f(z)=u(\Re z)+\mathrm{i}v(\Re z,\Im z)$ and satisfies $f(0)=0$. Then $f$ is linear.
\end{lemma}

\begin{proof}
Assume, as we may, that $k=1$.
The function $u$, being a real part of a holomorphic function, is pluriharmonic (see, for instance \cite{gunning}, page 102.)
Since the function $u$ depends only on its real part, at every $z\in \Omega$ the quadratic form $\mathcal{L}(z;c,c)$ for $c\in \mathbb{R}^m$ reduces to the (real) Hessian of $u$. In virtue of \eqref{pluriharmonic} the Hessian of $u$ vanishes on the whole $\Omega$. This means that $u$ is both convex and concave. As $u(0)=0$, we conclude that the function $u$ is linear.
It follows directly from the Cauchy-Riemann equations that the function $v$ is also linear as well.
\end{proof}

The forthcoming lemma describes the structure of linear free functions.

\begin{lemma} \label{L:linearmaps}
Let $F:D(E)\subseteq\mathcal{B}(E)^k \mapsto \mathcal{B}(E)$ be a linear free function where $(D(E))$ is a free set such that each $D(E)$ contains an open neighborhood of $0$ for each Hilbert space $E$. Then there exist $a_j \in \mathbb{C}$ for $j\in \mathbb{N}_k$ such that
\[
F(X)=\sum_{j=1}^k a_j\otimes X_j.
\]
\end{lemma}

\begin{proof}
Let $H\in D(E)$ so that $H=\Re H+i\Im H$, where $\Re H=(H+H^*)/2$ and $\Im H=(H-H^*)/(2i)$. Since $D(E)$ contains an open neighborhood of $0$ there exists an $r>0$ such that the open ball $B(0,r)\subseteq D(E)$. Thus, there exists an $\epsilon>0$ such that $\epsilon H\in B(0,r)$ and also $\epsilon\Re H, \epsilon\Im H\in B(0,r)$. By linearity of $F$ we conclude that
\[
\epsilon F(H)=F(\epsilon H)=F(\epsilon\Re H+i\epsilon\Im H)=F(\epsilon\Re H)+iF(\epsilon\Im H),
\]
so it is sufficient to determine $F(H)$ for all self-adjoint $H\in B(0,r)$. To this end, consider a self-adjoint operator $H=(H_1,\ldots, H_k)\in B(0,r)$ and observe that $e_j\otimes H_j\in B(0,r)$.
As each the $H_j$'s are unitary similar to some diagonal matrices, there exist unitaries $U_j\in \mathcal{B}(E)$ such that $H_j=U_j(\oplus_{m=1}^n d_{jm})U_j^*$ with some real numbers $d_{jm}$ for $m\in \mathbb{N}_n$ and $j\in \mathbb{N}_k$.
Denote $a_j:=F(e_j\otimes 1)$.
By linearity of $F$ and elementary properties of free functions, we deduce
\[
\begin{gathered}
F(H)=\sum_{j=1}^{k} U_j F(e_j \otimes (\oplus_{m=1}^n d_{jm})) U_j^* = \\
\sum_{j=1}^{k} U_j \left(\oplus_{m=1}^n F( e_j \otimes d_{jm}) \right) U_j^*=
\sum_{j=1}^{k} U_j \left(\oplus_{m=1}^n d_{jm}F( e_j \otimes 1) \right) U_j^* \\
=\sum_{j=1}^{k}F( e_j \otimes 1) U_j \left(\oplus_{m=1}^n d_{jm}\right) U_j^*=
\sum_{j=1}^{k}a_j\otimes H_j
\end{gathered}
\]
which completes the proof of the lemma.
\end{proof}

After all these preparations, we are in a position to prove the main result of the section.

\begin{proof}[Proof of Theorem~\ref{T:main}]
The 'if' part is apparent, so we are concerned with verifying the exciting 'only if' part.

Since $F$ is holomorphic, according to \cite{verbovetskyi} $F$ commutes with similarities. This yields that $F(0)$ is associated to the center of $\mathcal{B}(E)$. Therefore, there exists some $a_0\in \mathbb{C}$ such that $F(0)=a_0\otimes I$. Moreover, an application of Lemma~\ref{L:pluriharmonic} furnishes that the free function
\[
X\mapsto F(X)-a_0\otimes I
\]
is linear, whence the result follows directly from Lemma~\ref{L:linearmaps} and the fact that the above linear map is positive, for the latter see Lemma 2.3. in \cite{bearden}.
\end{proof}

\begin{remark}
So far we have studied functions in Definition~\ref{D:freeFunction} with $\mathcal{L}=\mathbb{C}$. Our results generalize to the setting when $\mathcal{L}$ is an arbitrary Hilbert space, since given a free function $F:D(E)\mapsto \mathcal{B}(\mathcal{L}\otimes E)$ for a domain $D(E)\subseteq\mathcal{B}(E)^k$, we can reduce to the case when $\mathcal{L}=\mathbb{C}$ by looking at the free function $F_h:D(E)\mapsto \mathcal{B}(E)$ defined as $F_h(X)=(h\otimes I)(F(X))$ where $h$ is a state on $\mathcal{B}(\mathcal{L})$. In this way the constants $a_j$ in Lemma~\ref{L:linearmaps} will become bounded linear operators in $\mathcal{B}(\mathcal{L})$ and furthermore $a_j\geq 0$ for $j\in \mathbb{N}_k$ accordingly in Theorem~\ref{T:main}.
\end{remark}

\begin{remark}
Another natural generalization is to consider more general domains $D(E)\subseteq\mathcal{B}(E)\otimes\mathcal{Z}$ for an operator space $\mathcal{Z}$ as in \cite{verbovetskyi}. Then for a given $X\in D(E)$, simultaneous unitary conjugation with $U\in\mathcal{B}(E)$ is to be understood as
\begin{equation*}
U^*XU:=(U^*\otimes I_\mathcal{Z})X(U\otimes I_\mathcal{Z})
\end{equation*}
and we get back to Definition~\ref{D:freeFunction} by choosing $\mathcal{Z}=\mathbb{C}^k$. Then Theorem~\ref{T:DependsOnFirstVar} and Corollary~\ref{C:DependsOnFirstVar} are still true and we can use Lemma~\ref{L:pluriharmonic} as well, since we may restrict to finite dimensional subspaces of the domain due to general properties of free holomorphic functions considered in \cite{verbovetskyi}. Then the statement of Theorem~\ref{T:main} reads that $F$ is real operator monotone if and only if it is affine linear with its nonconstant part being a (real) completely positive linear map. Here complete positivity is derived essentially from Proposition~\ref{P:FrechetRealPos}, since the Fr\'echet-derivative $DF(X)(Y)$ of an affine linear free function $F$ is a linear map of $Y$ that is independent of $X$. Thus the expression in Theorem~\ref{T:main} becomes
\begin{equation*}
F(X)=C\otimes I+\phi(X)
\end{equation*}
where $C\in\mathcal{B}(\mathcal{L})$ and $\phi:\mathcal{Z}\mapsto\mathcal{B}(\mathcal{L})$ is a completely positive linear map. This provides an alternative proof of the expression in Theorem~\ref{T:main} as well when $\mathcal{Z}=\mathbb{C}^k$ due to the structure theory of completely positive linear maps \cite{paulsen}.
\end{remark}

\subsection*{Acknowledgments}

The current research was partially supported by the National Research, Development and Innovation Office -- NKFIH Reg. No.'s K-115383 and K-128972, and by the Ministry of Human Capacities, Hungary through grant 20391-3/2018/FEKUSTRAT.

The work of Ga\'al was supported by the DAAD-Tempus PPP Grant 57448965.

The work of P\'alfia was supported by the National Research Foundation of Korea (NRF) grants founded by the Korea government (MEST) No.2015R1A3A2031159, No.2016R1C1B1011972 and No.2019R1C1C1006405.

\end{document}